\documentclass{article}
\usepackage{mystyle}
\usepackage[cp1251]{inputenc}
\usepackage{amssymb}
\usepackage{amsmath}
\usepackage{latexsym}
\usepackage{amsthm}
\newtheorem {theorem}{Theorem}
\newtheorem {definition}{Definition}
\newtheorem {prop}{Proposition}

\newtheorem {corollary}{Corollary}

\begin{document}
\title{Spectral properties of bipolar surfaces\\ to Otsuki tori}
\author{Mikhail A. Karpukhin}
\date{} 
\maketitle
\textsc{{Department of Geometry and Topology, Faculty of Mechanics and Mathematics, Moscow State University, Leninskie Gory, GSP-1, 
119991, Moscow, Russia}}

\smallskip

\textit{{and}}

\smallskip

\textsc{{Independent University of Moscow, Bolshoy Vlasyevskiy pereulok 11, 119002, Moscow, Russia}}

\smallskip

\textit{E-mail address:} \texttt{karpukhin@mccme.ru}
\begin{abstract}
The $i$-th eigenvalue $\lambda_i$ of the Laplace-Beltrami operator on a surface can be considered as a functional on 
the space of all Riemannian metrics of unit volume on this surface. Surprisingly only few examples of extremal metrics for these functionals are known. In the present paper a new countable family of extremal metrics on the torus is provided.
\end{abstract} 
\section*{Introduction.}
Let $M$ be a closed surface and $g$ be a Riemannian metric on $M$. Let us consider the associated Laplace-Beltrami operator
$\Delta$ acting on the space of smooth functions on~$M$,
$$
\Delta f = -\frac{1}{\sqrt{|g|}}\frac{\partial}{\partial x^i}\bigl(\sqrt{|g|}g^{ij}\frac{\partial f}{\partial x^j}\bigr).
$$
It is well-known that the spectrum of $\Delta$ is non-negative and consists only of eigenvalues, each eigenvalue has finite
 multiplicity and the eigenfunctions are smooth. Let us denote the eigenvalues of $\Delta$ by 
$$
0 = \lambda_0(M,g) < \lambda_1(M,g) \leqslant \lambda_2(M,g) \leqslant \lambda_3(M,g) \leqslant \ldots,
$$
where eigenvalues are written with multiplicities.

The eigenvalues possess the following property,
$$
\forall t>0\qquad \lambda_i(M,tg) = \frac{\lambda_i(M,g)}{t}, 
$$
Therefore, given a fixed surface $M$ one has $\sup\lambda_i(M,g) = +\infty$, where supremum is taken over the space of 
all Riemannian metrics on $M$. But if we consider supremum over the space of all Riemannian metrics on $M$ of unit area then
the question about the value of $\sup\lambda_i(M,g)$ becomes more interesting. In fact, in the case $\dim M = 2$ we can 
consider functionals
$$
\Lambda_i(M,g) = \lambda_i(M,g)\mathrm{Area}(M,g)
$$
unvariant under the transformation $g\mapsto tg$ and investigate their supremum over the space of all Riemannian metrics.

It is known that functionals $\Lambda_i(M,g)$ are bounded from above. Yang and Yau proved in the paper~\cite{YangYao} that
for an orientable surface $M$ of genus $\gamma$ the following inequality holds,
$$
\Lambda_1(M,g) \leqslant 8\pi(\gamma + 1).
$$ 
Moreover, Korevaar proved in the paper~\cite{Korevaar} 
that there exists a constant $C$ such that for any $i>0$ and any compact surface $M$ of genus $\gamma$ the following 
inequality holds, 
$$
\Lambda_i(M,g) \leqslant C(\gamma+1)i.
$$
However, Colbois and Dodziuk proved in the paper~\cite{ColboisDodziuk} that for any manifold $M$ of dimension $\dim M\geqslant 3$
the functional $\lambda_i(M,g)$ is not bounded on the space of Riemannian metrics $g$ on $M$ of unit volume.

The functional $\Lambda_i(M,g)$ depends continously  on the metric $g$, but this functional is not differentiable. 
However, Berger proved in the paper~\cite{Berger} that for an analytic family of metrics $g_t$ there exist the left and right 
derivatives
with respect to $t$. This is a motivation for the following definition, see the papers~\cite{ElSoufiIlias1,Nadirashvilli}.

\begin{definition} A Riemannian metric $g$ on a closed surface $M$ is called an extremal metric for the functional 
$\Lambda_i(M,g)$ if for any analytic deformation $g_t$ such that $g_0 = g$ the following inequality holds,
$$
\frac{d}{dt}\Lambda_i(M,g_t)\Bigl|_{t=0+} \leqslant 0 \leqslant \frac{d}{dt}\Lambda_i(M,g_t)\Bigl|_{t=0-}.
$$
\end{definition}

The detailed list of surfaces $M$ and values of index $i$ such that maximal or at least extremal metrics are known is quite
short and can be found in the introduction to the paper~\cite{PenskoiOtsuki}.

It turns out that extremal metrics are closely related to minimal submanifolds of the 
spheres. Let $M \looparrowright \mathbb{S}^n$ be a minimally
immersed submanifold of the unit sphere $\mathbb{S}^n \subset \mathbb{R}^{n+1}$. We denote by 
$\Delta$ the Laplace-Beltrami operator on $M$ associated with the induced metric $g$ on $M$. 
Let us introduce the eigenvalues counting function
$$
N(\lambda) = \#\{i|\lambda_i(M,g)<\lambda\}.
$$
This function is often called the Weyl's function. The following theorem provides a general approach to finding smooth extremal metrics.
\begin{theorem}[El Soufi and Ilias,~\cite{ElSoufiIlias2}] Let $M \looparrowright 
\mathbb{S}^n$ be a minimally
immersed submanifold of the unit sphere $\mathbb{S}^n \subset \mathbb{R}^{n+1}$. Then the metric induced on $M$ by the
 immersion is extremal for the functional $\Lambda_{N(2)}(M,g)$
\label{th1}
\end{theorem}

We also need to recall another result concerning minimal submanifolds of the sphere. This theorem can be found e.g. 
in the book~\cite{Kobayasi}.
\begin{theorem} Let $M \looparrowright \mathbb{S}^n$ be a minimally immersed submanifold of the unit sphere 
$\mathbb{S}^n \subset \mathbb{R}^{n+1}$. Then the restrictions $x^1|_M,\ldots,x^{n+1}|_M$ on $M$ of the standart coordinate
functions of $\mathbb{R}^{n+1}$ are eigenfunctions of the Laplace-Beltrami operator on $M$ with eigenvalue $\dim M$.
\label{th2}
\end{theorem}

Thus, it is possible to take an immersed minimal surface $M$ in the sphere, then compute $N(2)$ and deduce that the metric 
induced on $M$ by the immersion is extremal for $\Lambda_{N(2)}(M,g)$. This approach was successfully realized for the first 
time by Penskoi in the papers~\cite{PenskoiOtsuki,PenskoiLawson} for Otsuki tori and Lawson tau-surfaces.
Although, we should mention that Lapointe in the paper~\cite{Lapointe} used some of these ideas in investigation of bipolar 
surfaces to Lawson tau-surfaces. The work of Lapointe was inspired by the paper~\cite{JNP} where Jakobson, Nadirashvili and Polterovich 
proved that the metric on the Lawson bipolar surface $\tilde\tau_{3,1}$ is extremal for the functional $\Lambda_1(\mathbb{K}l,g)$.
Later, El Soufi, Giacomini and Jazar proved in the paper~\cite{EGJ} that this metric is the unique extremal metric.

In the present paper the extremality of the bipolar surfaces
to Otsuki tori is investigated. The definition of Otsuki tori and bipolar surfaces are given in Sections~\ref{OtsukiDef} and~\ref{BipolarDef} 
respectively. At this point it is sufficient to know that for every rational number $\dfrac{p}{q}$ 
such that $(p,q) = 1,\,\dfrac{1}{2}<\dfrac{p}{q}<\dfrac{\sqrt 2}{2},$ there exists a minimal immersed surface in 
$\mathbb{S}^4$ denoted by $\tilde O_{\frac{p}{q}}$. The main result of this paper is the following theorem.
\begin{theorem} The bipolar surface $\tilde O_{\frac{p}{q}}$ to an Otsuki torus is a torus. If $q$ is odd then the metric on 
$\tilde O_{\frac{p}{q}}$ induced by the immersion is extremal for $\Lambda_{2q+4p-2}(\mathbb{T}^2, g)$. If $q$ is even then 
the metric induced by the immersion on $\tilde O_{\frac{p}{q}}$ is extremal for $\Lambda_{q+2p-2}(\mathbb{T}^2,g)$.
\label{MainTheorem}
\end{theorem}
The paper is organized in the following way. The Otsuki tori and their bipolar surfaces are defined
in Sections~\ref{OtsukiDef} and~\ref{BipolarDef}. A convenient parametrization of bipolar surfaces is given in Section~\ref{BipolarParametrization}.
Section~\ref{proof} contains the proof of the main theorem.
\section{Bipolar surfaces to Otsuki tori.}
\label{BipSection}
\subsection{Reduction theorem for minimal submanifolds.} Let $M$ be a 
Riemannian manifold equipped with a metric $g'$ and $I(M)$ be its full
isometry group. Let $G\subset I(M)$ be a compact isometry group. Let us denote by $\pi$ 
the natural projection 
$\pi\colon M \to M/G$.

Denote by $M^*$ the union of all orbits of principal type, then $M^*$ is an open dense submanifold of $M$. The subset 
$M^*/G$ of $M/G$ is equipped with a natural Riemannian metric $g$ defined by the formula $g(X,Y) = g'(X',Y')$, where $X,Y$ are
tangent vectors at $x\in M^*/G$ and $X',Y'$ are tangent vectors at a point $x'\in\pi^{-1}(x)\subset M^*$ such that $X'$
and $Y'$ are orthogonal to the orbit $\pi^{-1}(x)$ and $d\pi(X')=X,\,d\pi(Y')=Y$.

Let $f\colon N \looparrowright M$ be a $G$-invariant immersed submanifold, i.e. a manifold equipped with 
an action of $G$ by isometries such that $g\cdot f(x) = f(g\cdot x)$ for any $x\in N$.
\begin{definition} A cohomogeneity of a $G$-invariant immersed submanifold $N$ is the number $\dim N - 
\nu$, where $\nu$ is the dimension of the orbits of principal 
type.
\end{definition}
Let us define for $x\in M^*/G$ a volume function $V(x)$ by the formula $V(x) = \mathrm{Vol}(\pi^{-1}(x))$. Also for each integer
$k \geqslant 1$ let us define a metric $g_k = V^{\frac{2}{k}}g$.
\begin{prop}[Hsiang, Lawson~\cite{HsiangLawson}] Let $f\colon N\looparrowright M^*$ be a $G$-invariant immersed submanifold of cohomogeneity $k$, and
let $M^*/G$ be equipped with the metric $g_k$. Then $f\colon N\looparrowright M^*$ is minimal if and only if 
$\bar f\colon N/G\looparrowright M^*/G$ is minimal.
\label{HsiangTh}    
\end{prop}
\subsection{Otsuki tori.} Otsuki tori\label{OtsukiDef} were introduced by Otsuki in the paper~\cite{Otsuki}. Let us 
recall the concise description by Penskoi from the paper~\cite{PenskoiOtsuki}. For more details see Section 1.2 of the paper~\cite{PenskoiOtsuki}. Consider the action of $SO(2)$ on 
the three-dimensional unit sphere $\mathbb{S}^3 \subset \mathbb{R}^4$ given by the formula
$$
\alpha\cdot(x,y,z,t) = (\cos\alpha x+\sin\alpha y, -\sin\alpha x +\cos\alpha y, z, t),
$$ 
where $\alpha \in [0,2\pi)$ is a coordinate on $SO(2)$. The space of orbits $\mathbb{S}^3/SO(2)$ is the closed half-sphere
$\mathbb{S}^2_+$,
$$
q^2+z^2+t^2=1, \qquad q\geqslant 0,
$$
where a point $(q,z,t)$ corresponds to the orbit $(q\cos\alpha,q\sin\alpha,z,t) \in \mathbb{S}^3$. The space of principal orbits
$(\mathbb{S}^3)^*/SO(2)$ is the open half sphere $\mathbb{S}^2_{>0} = \{(q,z,t)\in\mathbb{S}^2|q>0\}$. It is natural to introduce the spherical
 coordinates in the space of orbits,
$$
\left\{
   \begin{array}{rcl}
	t &=& \cos\nu\sin\lambda,\\
	z &=& \cos\nu\cos\lambda,\\
	q &=& \sin\nu\\
   \end{array}
\right.
$$
Since we look for minimal submanifolds of cohomogeneity 1, the Hsiang-Lawson's metric 
is given by the formula
\begin{equation}
V^2(d\nu^2 + \cos^2\nu d\lambda^2) = 4\pi^2\sin^2\nu(d\nu^2 + \cos^2\nu d\lambda^2). \label{otsukimetric}
\end{equation}
\begin{definition} An immersed minimal $SO(2)$-invariant two-dimensional torus in $\mathbb{S}^3$ such that its image
 by the projection $\pi\colon\mathbb{S}^3\to\mathbb{S}^3/SO(2)$ is a closed geodesics in $(\mathbb{S}^3)^*/SO(2)$ 
equipped with the metric (\ref{otsukimetric}) is called an Otsuki torus.
\end{definition}
The following proposition can be found in the paper~\cite{PenskoiOtsuki}.
\begin{prop} Except one particular case given by the equation $\psi = \frac{\pi}{4}$, Otsuki tori are in\label{PenskoiProp} 
one-to-one correspondence with rational numbers $\frac{p}{q}$ such that
$$
\frac{1}{2}<\frac{p}{q}<\frac{\sqrt{2}}{2},\qquad p,q>0,\,(p,q) = 1.
$$
\end{prop}
\begin{definition} By $O_{\frac{p}{q}}$ we denote the Otsuki torus corresponding to 
$\frac{p}{q}$. Following the paper~\cite{PenskoiOtsuki} we reserve the term "Otsuki tori"
for the tori $O_{\frac{p}{q}}$.
\end{definition} 
In order to fix notations we give a sketch of the proof of Proposition~\ref{PenskoiProp}. 
\begin{proof} Let us use the standard notation for the coefficients of the metric (\ref{otsukimetric}),
$$
E = 4\pi^2\sin^2\nu, \qquad G = 4\pi^2\sin^2\nu\cos^2\nu.
$$
The equation of geodesics for $\ddot\lambda$ reads
$$
\ddot\lambda + \frac{1}{G}\frac{\partial G}{\partial\nu}\dot\nu\dot\lambda = 0.
$$
Hence, $2\pi c = G\dot\lambda$ is an integral of motion and 
\begin{equation}
\label{lambda}
\dot\lambda = \frac{c}{2\pi\cos^2\nu\sin^2\nu}.
\end{equation}
As we know the velocity vector of a geodesic has a constant length. Suppose this length equals 1. Then
\begin{equation}
\label{psi}
E\dot\nu^2 + G\dot\lambda^2 = 1 \Leftrightarrow \dot\nu^2 = \frac{\sin^2\nu\cos^2\nu - c^2}{4\pi^2\sin^4\nu\cos^2\nu}
\end{equation} 
This implies $\sin^2\nu\cos^2\nu - c^2 \geqslant 0$ and $\sin^2\nu\cos^2\nu=c^2$ iff $\dot\nu = 0$. 

Since the point corresponding to $\nu=0$ does not belong to $(\mathbb{S}^3)^*/SO(2)$, there exists a minimal value $a$ of the coordinate $\nu$
on a geodesic. Therefore $c = \pm\sin a\cos a$ and the geodesics are situated in the annulus 
$a\leqslant \nu \leqslant \frac{\pi}{2} - a$. We choose a natural parameter $t$ such that $\nu(0) = a$.

Equations (\ref{lambda}) and (\ref{psi}) imply
$$
\frac{d\nu}{d\lambda} = \pm\cos\nu\frac{\sqrt{\sin^2\nu\cos^2\nu - \sin^2 a\cos^2 a}}{\sin a\cos a}.
$$
The right hand side of this equation equals 0 only at $\nu = a$ and $\nu = \frac{\pi}{2} - a$.

Let us denote by $\Omega(a)$ the distance between the value of $\lambda$ corresponding to $\nu = a$ and the closest to it value 
of $\lambda$ corresponding to $\nu = \frac{\pi}{2}-a$. It is clear that
$$
\Omega(a) = \sin a\cos a\int\limits_a^{\frac{\pi}{2}-a}\frac{d\nu}{\cos\nu\sqrt{\sin^2\nu\cos^2\nu - \sin^2 a\cos^2 a}}.
$$

The geodesic is closed iff $\Omega(a) = \dfrac{p}{q}\pi$. The rest of the proof follows from properties of the function
 $\Omega(a)$, see the paper~\cite{Otsuki},
\begin{itemize}
\item[1)] $\Omega(a)$ is continuous on $\left(0,\dfrac{\pi}{4}\right]$,
\item[2)] $\lim\limits_{a\to 0+}\Omega(a) = \dfrac{\pi}{2}$ and $\Omega\left(\dfrac{\pi}{4}\right) = \dfrac{\pi}{\sqrt{2}}$.
\end{itemize}
\end{proof}

The Otsuki tori 
$O_{\frac{p}{q}}$ are minimally immersed into $\mathbb{S}^3$ by 
\begin{equation*}
\begin{split}
I_a&\colon [0,2\pi)\times [0,\tilde t) \to \mathbb{R}^4 \\ I_a&(\alpha,t) = (\cos\alpha\sin\nu (t), \sin\alpha\sin\nu (t), 
\cos\nu (t)\cos\lambda (t),\cos\nu (t)\sin\lambda (t)),
\end{split}
\end{equation*}
where $\Omega(a) = \frac{p}{q}\pi$ and $t$ is a natural parameter on the corresponding closed geodesic $\pi(O_\frac{p}{q})$ such that 
$\min\limits_t\nu(t) = \nu(0) = a$ and $\tilde t$ is the length of this geodesic.
\subsection{Construction of bipolar surfaces.} 
Following the papers~\cite{Kenmoutsu, Lawson}, we define the surface $\tilde O_{\frac{p}{q}}$ bipolar to $O_{\frac{p}{q}}$ as an 
exterior product of $I$ and $I^*$, where $I^*$ is a unit vector normal to the torus $O_{\frac{p}{q}}$ and tangent to 
$\mathbb{S}^3$. By a straightforward computation one obtains
\begin{equation*}
\begin{split}
I^*_a = 2\pi\sin\nu(\dot\lambda\cos^2\nu\cos\alpha&, \dot\lambda\cos^2\nu\sin\alpha, 
\dot\nu\sin\lambda - \dot\lambda\cos\nu\sin\nu\cos\lambda,\\& -\dot\nu\cos\lambda  -
\dot\lambda\cos\nu\sin\nu\sin\lambda),
\end{split}
\end{equation*} 
where the dot denotes the derivative with respect to $t$, and
\begin{equation}
\begin{split}
&I_a\wedge I^*_a = \\&2\pi\sin\nu(0,\cos\alpha(\dot\lambda\cos\lambda\cos\nu-\dot\nu\sin\lambda\sin\nu), 
\cos\alpha(\dot\lambda\sin\lambda\cos\nu+\dot\nu\cos\lambda\sin\nu),\\
&\sin\alpha(\dot\lambda\cos\lambda\cos\nu-\dot\nu\sin\lambda\sin\nu),
 \sin\alpha(\dot\lambda\sin\lambda\cos\nu+\dot\nu\cos\lambda\sin\nu), \dot\nu\cos\nu).
\end{split} 
\label{bipolarimmersion1}
\end{equation}

The paramerized surface $I_a\wedge I^*_a$ is a minimal (see a proof in the paper~\cite
{Lawson}) immersed submanifold in the equator $\mathbb{S}^4 \subset \mathbb{S}^5$. But formula 
(\ref{bipolarimmersion1}) is inconvenient. In the next section another parametrization of $\tilde O_{\frac{p}{q}}$ is 
proposed. 
\label{BipolarDef}
\subsection{Paramerization of $\tilde O_{\frac{p}{q}}$.} Let us now apply the Hsiang-Lawson's reduction theorem 
(Proposition~\ref{HsiangTh}) in the case\label{BipolarParametrization} 
of $M=\mathbb{S}^4$ and $G=SO(2)$. Let $x,y,z,u,v$ be the standard coordinates in 
$\mathbb{R}^5$ and $\mathbb{S}^4$ be the standard unit sphere in $\mathbb{R}^5$. Let us consider 
an action of $SO(2)$ given by the formula
$$
\alpha\cdot(x,y,z,u,v) = (\cos\alpha x-\sin\alpha y, \sin\alpha x +\cos\alpha y,
 \cos\alpha z-\sin\alpha u, \sin\alpha z +\cos\alpha u, v),
$$ 
where $\alpha\in[0,2\pi)$ is a coordinate on $SO(2)$. 

The principal orbits are circles of radius $\sqrt{x^2+y^2+z^2+u^2}$, 
the exceptional orbits are the poles $N = (0,0,0,0,1)$ and $S = (0,0,0,0,-1)$. It is easy to see, that for each principal orbit there are exactly two points
on the equatorial sphere $\mathbb{S}^3$ of the unit sphere $\mathbb{S}^4$ given by the equation $y=0$.
Therefore, the space of orbits $(\mathbb{S}^4)^*/SO(2)$ can be identified with the quotient of this equatorial sphere $\mathbb{S}^3$ by the action of $\mathbb{Z}_2$ given by 
\begin{equation}
\label{action}
\sigma(x,0,z,u,v) = (-x,0,-z,-u,v),
\end{equation}
where $\sigma$ is the nontrivial element of $\mathbb{Z}_2$. Let us call 
the equatorial sphere given by the equation $y=0$ a generalized space of orbits. Let us denote by $p$ the quotient map from the generalized space of orbits
to the space of orbits,
$$
p\colon \mathbb{S}^3\backslash\{N,S\} \to (\mathbb{S}^4)^*/SO(2).
$$
Let us denote by $\tilde\pi$ the natural projection of $(\mathbb{S}^4)^*$ onto the space of orbits.

Let $g_1$ be the Hsiang-Lawson's metric on the space of orbits. The preimage $p^{-1}(s)$ of a closed geodesic $s$ in the space of orbits is either 
 a closed geodesic $\gamma$ in $(\mathbb{S}^3\backslash\{N,S\},p^*g_1)$ such that $\sigma\gamma = \gamma$, or a pair of closed 
geodesics $\{\gamma_1,\gamma_2\}$ in $(\mathbb{S}^3\backslash\{N,S\},p^*g_1)$ such that $\sigma\gamma_1 = \gamma_2$.
Thus, each geodesic in the space of orbits is the image $p(\gamma)$ of some geodesic $\gamma$ in the generalized space of orbits.

It is useful to introduce the spherical coordinates in the generalized space of orbits,
$$
\left\{
   \begin{array}{rcl}
	x &=& \cos\varphi\sin\theta,\\
	z &=& \cos\varphi\cos\theta\cos\rho,\\
	u &=& \cos\varphi\cos\theta\sin\rho,\\
	v &=& \sin\varphi.\\
   \end{array}
\right.
$$
Then the pullback of the volume function to the generalized space of orbits is given by the formula $V(\varphi,\theta,\rho) = 
2\pi\sin\varphi$.

These coordinates induce coordinates on $\mathbb{S}^4$ by the following formulae,
$$
\left\{
    \begin{array}{rcl}
	x &=& \cos\alpha\cos\varphi\sin\theta,\\
	y &=& \sin\alpha\cos\varphi\sin\theta,\\
	z &=& \cos\alpha\cos\varphi\cos\theta\cos\rho - \sin\alpha\cos\varphi\cos\theta\sin\rho,\\
	u &=& \sin\alpha\cos\varphi\cos\theta\cos\rho + \cos\alpha\cos\varphi\cos\theta\sin\rho,\\
	v &=& \sin\varphi,\\
    \end{array}
\right.	
$$
where $\alpha\in[0,\pi)$.
The metric on $\mathbb{S}^4$ is given by the formula,
\begin{equation}
\label{tildeg1}
\cos^2\varphi d\alpha^2 + d\varphi^2 + \cos^2\varphi d\theta^2 + \cos^2\varphi\cos^2\theta 
(d\alpha d\rho + d\rho^2),
\end{equation}
and the induced metric on the generalized space of orbits is given by the formula,
$$ 
g = d\varphi^2 + \cos^2\varphi d\theta^2 + \cos^2\varphi\cos^2\theta\sin^2\theta d\rho^2.
$$

Minimal $SO(2)$-invariant submanifolds of cohomoheneity 1 of the sphere $\mathbb{S}^4$ correspond to closed 
geodesics in the space of orbits $(\mathbb{S}^4)^*/SO(2)$. According to the discussion at the 
beginning of this section, in order to find these submanifolds it is sufficient to find closed geodesics in $\mathbb{S}^3\backslash\{N,S\}$ equipped with the 
metric
$$
g_1 = V^2g = 4\pi^2\cos^2\varphi(d\varphi^2 + \cos^2\varphi d\theta^2 + \cos^2\varphi\cos^2\theta\sin^2\theta d\rho^2).
$$
Indeed, for any closed geodesic $s$ in the space of orbits there exists a closed geodesic $\gamma$ in the generalized space of orbits such that $p(\gamma) = s$. Therefore, the minimal submanifold $\tilde\pi^{-1}(s)$ coincides with the submanifold $\tilde\pi^{-1}(p(\gamma))$. Moreover, the image by $p$ of a geodesic in the generalized space of orbits is a geodesic in the space of orbits. Hence, the set of submanifolds $\tilde\pi^{-1}(p(\gamma))$ is exactly the set of minimal $SO(2)$-invariant submanifolds of cohomoheneity 1.  

Since the coefficients of the metric $g_1$ do not depend on $\rho$, the 2-dimensional sphere 
defined by $\rho = 0$ is the totally geodesic 2-sphere equipped with the metric
$$
\tilde g_1 = 4\pi^2\cos^2\varphi(d\varphi^2 + \cos^2\varphi d\theta^2).
$$

Let us now look for minimal submanifolds of the special type. Consider the sphere $\mathbb{S}^2\subset\mathbb{S}^4$ defined by $y = 0,\,\rho = 0$.
Then for a closed geodesic $\gamma(t) = (\varphi(t),\theta(t))$ in the space $(\mathbb{S}^2\backslash\{N,S\},\tilde g_1)$ one
has the corresponding immersed minimal submanifold $\tilde\pi^{-1}(p(\gamma))$ in $\mathbb{S}^4$. The immersion $J$ is given by 
the formula
\begin{equation}
\label{bipolarimmersion2}
    \begin{split}
	x &= \cos\alpha\cos\varphi(t)\sin\theta(t),\\
	y &= \sin\alpha\cos\varphi(t)\sin\theta(t),\\
	z &= \cos\alpha\cos\varphi(t)\cos\theta(t),\\
	u &= \sin\alpha\cos\varphi(t)\cos\theta(t),\\
	v &= \sin\varphi(t),
    \end{split}
\end{equation}
where $\alpha\in[0,2\pi)$. 
\begin{prop} The set of bipolar surfaces $\tilde O_{\frac{p}{q}}$ coincides with the set of minimal surfaces
$\tilde\pi^{-1}(p(\gamma))\subset\mathbb{S}^4$, where $\gamma$ is a closed geodesic in the space 
$(\mathbb{S}^2\backslash\{N,S\},\tilde g_1)$.
\label{parametrization} 
\end{prop}
\begin{proof} In the same way as in the proof of Proposition \ref{PenskoiProp}, one obtains
\begin{equation}
\label{theta}
\dot\theta = \frac{\cos^2 b}{2\pi\cos^4\varphi} 
\end{equation}
and
\begin{equation}
\label{phi}
\dot\varphi = \pm\frac{\sqrt{\cos^4\varphi - \cos^4 b}}{2\pi\cos^3\varphi}.
\end{equation}
By $J_b(\alpha,t)$ denote the immersion of the minimal submanifold corresponding to 
a geodesic such that the minimal value of $\varphi(t)$ equals to $b$, where $t$ is a natural 
parameter on the geodesic $\tilde\pi(J_b)$.
Let us show that for any point $\gamma(t)$ on the geodesic $\pi(I_a)$ there exists a neighbourhood $U$ of a point $t\in \mathbb{R}/(l\mathbb{Z})$, where $l$ is the length of the geodesic $\pi(I_a)$, 
and a function $\tau(t)$ defined on $U$, such that $I_a\wedge I_a^*(\alpha,\tau(t)) = J_{b(a)}(\alpha,t)$, where 
$\cos^4b(a) = 4\sin^2a\cos^2a$. Comparing equations (\ref{bipolarimmersion1}) and (\ref{bipolarimmersion2}) one obtains
\begin{equation}
\label{1}
\sin\varphi(t) = 2\pi\dot\nu(\tau(t))\cos\nu(\tau(t))\sin\nu(\tau(t)).
\end{equation}

Let us consider the case $\dot\nu > 0$ and $\dot\varphi > 0$. On the one hand, using formula (\ref{psi}), one has
\begin{equation}
\label{2}
\sin\varphi(t) = \frac{\sqrt{\cos^2\nu(\tau(t))\sin^2\nu(\tau(t)) - c^2}}{\sin\nu},
\end{equation}
where $c = \sin a\cos a$. Applying $\dfrac{d}{dt}$ to equation (\ref{2}) and using formula (\ref{psi}) one obtains
$$
\dot\varphi(t)\cos\varphi(t) = \dot\tau(t)\frac{c^2 - \sin^4\nu(\tau(t))}{2\pi\sin^4\nu(\tau(t))\cos\nu(\tau(t))}.
$$
On the other hand, combining equations (\ref{phi}) and (\ref{1}) one has the following formula,
$$
\dot\varphi(t)\cos\varphi(t) = \frac{c^2 - \sin^4\nu(\tau(t))}{2\pi(\sin^4\nu(\tau(t)) + c^2)}.
$$ 
Therefore, one obtains a differential equation for $\tau(t)$,
\begin{equation}
\label{tau}
\dot\tau = \frac{\sin^4\nu(\tau)\cos\nu(\tau)}{\sin^4\nu(\tau) + c^2}.
\end{equation}
Let $\tau(t)$ be a solution of this equation. Comparing equations (\ref{bipolarimmersion1}) and 
(\ref{bipolarimmersion2}) one has the following formulae,
\begin{equation}
\label{3}
\begin{split}
\cos\varphi(t)\sin\theta(t) &= 2\pi\sin\psi(\tau(t))(\dot\lambda\cos\lambda\cos\psi - \dot\nu\sin\lambda\sin\nu)(\tau(t)),\\
\cos\varphi(t)\cos\theta(t) &= 2\pi\sin\psi(\tau(t))(\dot\lambda\sin\lambda\cos\psi + \dot\nu\cos\lambda\cos\nu)(\tau(t)).
\end{split}
\end{equation}

One should prove that for a function $\theta(t)$ defined by implicit formulae (\ref{3}) differential equation (\ref{theta}) holds.
 This can be shown by a straightforward calculation (we omit it in order to shorten the paper).
This completes the proof.
\end{proof}   
\subsection{Properties of the new parametrization.}
Let us denote by $t_0$ the length of the geodesic $\tilde\pi(\tilde O_{\frac{p}{q}})$ with respect to the metric $\tilde g_1$. As coordinates on the torus $\tilde O_{\frac{p}{q}}$ we take the parameter $\alpha\in [0,2\pi)$ on $SO(2)$ and a 
natural parameter $t \in [0,t_0)$ on the geodesic $\tilde\pi(\tilde O_{\frac{p}{q}}) = (\varphi(t),\theta(t))$ such that $\min\limits_t\varphi(t) = \varphi(0) = b$.
\label{Properties}
\begin{prop}
The function $\sin\varphi(t)$ has exactly $2q$ zeroes on $[0,t_0)$, the functions $\cos\theta(t)$ and $\sin\theta(t)$ both
have exactly $2p$ zeroes on $[0,t_0)$. If $q$ is even then the immersion $J_b$ is invariant under the transformation 
$(\alpha,t) \mapsto \left(\alpha+\pi,t+\dfrac{t_0}{2}\right)$. The immersion $J$ is not invariant under any other
transformations.\label{zeroam}  
\end{prop}
\begin{proof}
Let us remark that the immersions $I_a$ and $J_b$ are well-defined even if the corresponding geodesics are not closed. 
We proved in Proposition~\ref{parametrization} that the bipolar surface to $I_a$ corresponds to the geodesic $\tilde\pi(J_b)$,
where
\begin{equation}
\label{b(a)}
\cos^4b = 4\sin^2a\cos^2a.
\end{equation}
Hence, $\pi(I_a)$ is closed iff $\tilde\pi(J_b)$ is closed.

According to formula~(\ref{bipolarimmersion1}), the geodesic $\tilde\pi(J_b)$ admits another parametrization in terms of 
$\lambda(s)$ and $\nu(s)$, where $s$ is a natural parameter on $\pi(I_a)$. It is easy to see that this parametrization is one-to-one outside of self-intersection points, i.e. for the map
$$
\beta(s)=2\pi\sin\nu(\dot\lambda\cos\lambda\cos\nu-\dot\nu\sin\lambda\sin\nu, 
\dot\lambda\sin\lambda\cos\nu+\dot\nu\cos\lambda\sin\nu,\dot\nu\cos\nu),
$$ 
where $s\in[0,\tilde t)$, there is no point $\tilde s$ such that $\beta([0,\tilde s)) = \beta([\tilde s,\tilde t)) = \beta([0,\tilde t))$. Indeed, 
since $a\leqslant\nu\leqslant\dfrac{\pi}{2}-a$,
the last coordinate is equal to zero only at zeroes of $\dot\nu(s)$, i.e at $s_d=\dfrac{\tilde td}{2q}$, where $d=0,1,\ldots, 
2q-1$ and $\lambda(s_d) = \dfrac{pd}{q}\pi$. Hence, there exists $d=0,1,\ldots, 
2q-1$ such that $\tilde s = s_d$. Moreover, $\nu(s_d) = a$ if $d$ is even 
and $\nu(s_d) = \dfrac{\pi}{2} - a$ if $d$ is odd. The value of $2\pi\dot\lambda\sin\nu\cos\nu = \dfrac{\sin 
a}{\sin\nu\cos\nu}$ is equal to $\dfrac{1}{\cos a}$ for each point $s_d$. Therefore, $\cos\lambda(s_d)=\cos\lambda(0) = 1$ and 
$\sin\lambda(s_d)=\sin\lambda(0)=0$. This holds for $s_0=0$ and possibly for $s_q$. In the latter case $\nu(s_q) = 
\dfrac{\pi}{2} - a$ and $(2\pi\dot\nu\sin\nu\cos\nu)(s_q+\varepsilon)<0$ for sufficiently small $\varepsilon$. This 
contradicts the fact that $(2\pi\dot\nu\sin\nu\cos\nu)(\varepsilon)>0$.

The previous statement implies that the function $\sin\varphi(t)$ has the same quantity of zeroes as
$2\pi\dot\nu(t)\sin\nu(t)\cos\nu(t)$, i.e. $\sin\varphi(t)$ has exactly $2q$ zeroes.

Let us introduce a function analogous to $\Omega(a)$. The function $\Xi(b)$ equals the distance between the nearest points
on the geodesic $\tilde\pi(J_b)$ with $\varphi = b$ and $\varphi = -b$,
$$
\Xi(b) = \cos^2b\int\limits_{b}^{-b}\frac{1}{\cos\varphi\sqrt{\cos^4\varphi-\cos^4b}}\,d\varphi.
$$  
In Section~\ref{dxi} the following proposition is proved.
\begin{prop}
The function $\Xi(b)$ is increasing and continuous on the interval $\left(-\dfrac{\pi}{2}, 0\right)$. The following equality holds, $\lim\limits_{b\to 0-}\Xi(b) = \dfrac{\sqrt{2}}{2}\pi$.
\label{xi}
\end{prop}
The geodesic $\tilde\pi(J_b)$ is closed iff $\Xi(b) = \dfrac{r}{s}\pi$, where $r,s\in\mathbb{Z}_{\geqslant 0}$. Without loss of generality one can assume that $(r,s) = 1$. Since $\sin\varphi(t)$ has $2q$ zeroes, one has
$s=q$. According to formula~(\ref{b(a)}), the function $b(a)$ increases as $a$ increases. So, we have two increasing continuous functions 
$\Omega(a)$ and $\Xi(b(a))$ such that their values at the point $a = \dfrac{\pi}{4}$ coincide and
\begin{equation}
\label{condition}
\Omega(a) = \frac{p}{q}\pi\qquad \Leftrightarrow \qquad \Xi(b(a)) = \frac{r}{q}\pi.
\end{equation}
We claim that such two functions coincide. Indeed, let us introduce the following sets,
\begin{equation*}
\begin{split}
& A_\Omega(s) = \left\{\frac{p}{s}\pi\,|\,\, \Omega(0)<\frac{p}{s}\pi<\Omega\left(\frac{\pi}{4}\right),\,(p,s) = 1\right\} \\
& A_\Xi(s) = \left\{\frac{p}{s}\pi\,|\,\, \Xi(b(0))<\frac{p}{s}\pi<\Xi\left(b\left(\frac{\pi}{4}\right)\right),\,(p,s) = 1\right\}.
\end{split}
\end{equation*} 
On the one hand, condition (\ref{condition}) implies that $|A_\Omega(s)| = |A_\Xi(s)|$ for any $s$. On the other hand, suppose that $\Omega(0) \ne \Xi(b(0))$. Then for a sufficiently large $s$ one has $|A_\Omega(s)| \ne |A_\Xi(s)|$. This observation leads to a contradiction, hence $\Omega(0) = \Xi(b(0)) = \dfrac{1}{2}\pi$. Then $A_\Omega(s) = A_\Xi(s)$ and we denote this set simply by $A(s)$.
Let us consider the inverse functions $f = \Omega^{-1}$ and $g = (\Xi\,\circ\, b)^{-1}$. These functions are monotonous and continous on the interval 
$\left(\dfrac{1}{2}\pi,\dfrac{\sqrt{2}}{2}\pi\right)$. Condition (\ref{condition}) means that for any $s$ one has 
$f(A(s)) = g(A(s))$.
By monotonicity of $f$ and $g$, $f(x) = g(x)$ for any $x\in A(s)$. Therefore $f$ and $g$ coincide on the dense subset $\bigcup\limits_{s}A(s)$ of interval, hence by continuity $f(x)\equiv g(x)$ and $\Omega(a) = \Xi(b(a))$. 
Then since $\Omega(a) = \Xi(b(a))$, the functions $\cos\theta(t)$ and $\sin\theta(t)$ have $2p$ zeroes. 

Since each orbit has exactly two intersection points with the generalized space of orbits, the immersion $J_b$ is invariant under some tranformation if and only 
if the corresponding geodesic $\gamma=\tilde\pi(\tilde O_{\frac{p}{q}})$
 is invariant under the action of $\mathbb{Z}_2$ given by the formula (\ref{action}). 
This means that $\mathrm{Im}\gamma$ contains the point $(b,\pi)$. According to the first statement of this proposition,
if $(\varphi(t_1),\theta(t_1))=\gamma(t_1)$ and $\varphi(t_1) = b$ then $\theta(t_1) = \dfrac{2kp}{q}\pi$, where 
$k = 0,1,\ldots,q-1$. Hence, $\theta(t_1)$ can be equal to $(2l+1)\pi$ if and only if $q\equiv 0 \mod 2$ and $k = \dfrac{q}{2}$.
This implies that $t_1 = \dfrac{t_0}{2}$. Since the map $(\varphi,\theta)\mapsto(\varphi,\theta+\pi)$ is an isometry of the orbit 
space, $J_b$ is invariant under the transformation $(\alpha,t)\mapsto\left(\alpha+\pi,t+\dfrac{t_0}{2}\right)$.  
\end{proof}
\section{Proof of the Theorem}
\label{proof}
\subsection{Relation to the theory of periodic Sturm-Liouville problems.} In this section the eigenvalue counting problem for 
the Laplace-Beltrami operator on the bipolar 
Otsuki torus $\tilde O_{\frac{p}{q}}$ is reduced to the same problem for the periodic Sturm-Liouville operator.
\begin{prop} Let $\tilde O_{\frac{p}{q}}$ be a bipolar surface to an Otsuki torus $O_{\frac{p}{q}}$ parametrized 
by immersion $J_b(\alpha,t)$ as in Section~\ref{Properties}. Then the corresponding Laplace-Beltrami operator is given by the formula
\begin{equation}
\label{laplace}
\Delta f = -\frac{1}{\cos^2\varphi(t)}\frac{\partial^2 f}{\partial\alpha^2} - \frac{\partial}{\partial t}
\left( 4\pi^2\cos^2\varphi(t)\frac{\partial f}{\partial t} \right ).
\end{equation}
\end{prop}
\begin{proof} The metric on the sphere $\mathbb{S}^4$ is given by formula (\ref{tildeg1}). Since 
$\rho = 0$, the metric on $\tilde O_{\frac{p}{q}}$ is given by the formula
$$
\cos^2\varphi(t)d\alpha^2 +(\dot\varphi(t)^2 + \dot\theta(t)^2\cos^2\varphi(t))dt^2.
$$ 
But the length of the velocity vector of $\tilde\pi(O_{\frac{p}{q}})$ is equal to 1, therefore,
$$
4\pi^2\cos^2\varphi(t)(\dot\varphi(t)^2 + \dot\theta(t)^2\cos^2\varphi(t)) = 1.
$$
Hence the metric on $\tilde O_{\frac{p}{q}}$ equals to
\begin{equation}
\label{h}
h = \cos^2\varphi(t)d\alpha^2 + \frac{1}{4\pi^2\cos^2\varphi(t)}dt^2
\end{equation}
and formula (\ref{laplace}) could be obtained by a direct calculation.
\end{proof}
\begin{prop} A\label{prop1} number $\lambda$ is an eigenvalue of $\Delta$ if and only if there exists $l \in \mathbb{Z}_{\geqslant 0}$
and an eigenvalue $\lambda(l)$ of the following periodic Sturm-Liouville problem
\begin{equation}
\label{sturm}
\begin{split}
&\frac{d}{dt}\left(4\pi^2\cos^2\varphi(t)\frac{dh(t)}{dt}\right) + \left(\lambda - \frac{l^2}{\cos^2\varphi(t)}\right)h(t) = 0, \\
&h(t+t_0) \equiv h(t),
\end{split}
\end{equation}
such that $\lambda(l) = \lambda$.
\end{prop}
\begin{proof} Let us remark that $\Delta$ commutes with $\dfrac{\partial}{\partial\alpha}$. It follows that $\Delta$ has a basis 
of eigenfunctions of the form $h(l,t)\cos(l\alpha)$ and $h(l,t)\sin(l\alpha)$. Substituting these eigenfunctions into the formula 
$\Delta f = \lambda f$ one obtains equation (\ref{sturm}). Since $f(\alpha+2\pi,t) \equiv f(\alpha, t+t_0) \equiv 
f(\alpha,t)$, one has $l\in\mathbb{Z}$ and the boundary condition in formula (\ref{sturm}).
\end{proof}

The equation (\ref{sturm}) is written in the classical form of the periodic Sturm-Liouville problem, and the following
 proposition holds, see e.g. book~\cite{KoddingtonLevinson}.
\begin{prop} Consider a periodic Sturm-Liouville problem in the form 
\begin{equation}
\label{Rayleigh}
-\frac{d}{dt}\left(p(t)\frac{d}{dt}h(t)\right) + q(t)h(t) = \lambda h(t),
\end{equation}
where $p(t),q(t)>0$ and $p(t+t_0)\equiv p(t),q(t+t_0)\equiv q(t)$. Let us denote by 
$\lambda_i$ and $h_i(t)$ ($i = 0,1,2,\ldots$) the eigenvalues and eigenfunctions of the problem (\ref{Rayleigh})
with the periodic boundary conditions
\begin{equation}
\label{Periodic}
h(t+t_0)\equiv h(t).
\end{equation}
Let us also denote by $\tilde\lambda_i$ and $\tilde h_i(t)$ ($i=1,2,\ldots$) the eigenvalues and 
eigenfunctions of the problem (\ref{Rayleigh}) with antiperiodic boundary conditions
\begin{equation}
\label{Antiperiodic}
h(t+t_0)\equiv -h(t).
\end{equation}
Then the following inequalities hold, 
$$
\lambda_0<\tilde\lambda_1\leqslant\tilde\lambda_2<\lambda_1\leqslant\lambda_2<\tilde\lambda_3\leqslant\tilde\lambda_4<\lambda_3\leqslant\lambda_4<\ldots
$$
For $\lambda = \lambda_0$ there exists a unique (up to multiplication by a non-zero constant) eigenfunction $h_0(t)$.
If $\lambda_{2i+1}<\lambda_{2i+2}$ for $i\geqslant 0$ there is a unique (up to multiplication by a non-zero constant) eigenfunction $h_{2i+1}(t)$ with eigenvalue
$\lambda_{2i+1}$ of multiplicity 1 and there is a unique (up to multiplication by a non-zero constant) eigenfunction $h_{2i+2}(t)$ with eigenvalue $\lambda_{2i+1}$
of multiplicity one. If $\lambda_{2i+1} = \lambda_{2i+2}$ then there is two-dimensional eigenspace spaned by
$h_{2i+1}(t)$ and $h_{2i+2}(t)$ with eigenvalue $\lambda = \lambda_{2i+1} = \lambda_{2i+2}$
 of multiplicity two. The same holds in case $\tilde\lambda_{2i+1}<\tilde\lambda_{2i+2}$ 
and $\tilde\lambda_{2i+1} = \tilde\lambda_{2i+2}$\label{SturmProp}

The eigenfunction $h_0(t)$ has no zeros on $[0,t_0)$. The eigenfunctions $h_{2i+1}(t)$ 
and $h_{2i+2}(t)$ each have
exactly $2i+2$ zeros on $[0,t_0)$. The eigenfunctions $\tilde h_{2i+1}(t)$ 
and $\tilde h_{2i+2}(t)$ each have exactly $2i+1$ zeros on $[0,t_0)$. 
\end{prop}
\begin{corollary} Let $h_i(l,t)$ and $\lambda_i(l)$ be the $i$-th eigenfunction and the $i$-th eigenvalue of problem~(\ref{sturm}) 
for a fixed $l$.
Then the eigenspace of the Laplace-Beltrami operator $\Delta$ with eigenvalue $\lambda$ has a basis consisting
of functions of the form\label{basis}
$$
h_i(l,t)\cos(l\alpha),
$$
where $l\in\mathbb{Z}_{\geqslant 0}$ and there exists $i$ such that $\lambda_i(l) = \lambda$, and
$$
h_i(l,t)\sin(l\alpha),
$$
where $l\in\mathbb{N}$ and there exists $i$ such that $\lambda_i(l) = \lambda$.
\end{corollary}    
\begin{proof} The statement follows from Propositions \ref{prop1} and \ref{SturmProp} for a fixed $l$.
\end{proof}
\subsection{Rayleigh quotient.}
Let us now investigate properties of eigenvalues $\lambda_i(l)$ as functions of $l$. One of the most efficient tools 
for this investigation is a Rayleigh quotient. The Rayleigh quotient for the problem~(\ref{Rayleigh}) is defined by 
the following formula,
$$
R[v] = \frac{\int_0^{t_0}p(t)\dot v^2 + q(t)v^2 dt}{\int_0^{t_0}v^2 dt}. 
$$

The following proposition can be found e.g. in the book~\cite{Henrot}. 
\begin{prop}[Variational principle] For the eigenvalue $\lambda_0$ of the problem (\ref{Rayleigh}) with the boundary 
condition (\ref{Periodic}) one has 
$$
\lambda_0 = \inf_v R[v], 
$$
where infimum is taken over the space of $t_0$-periodic functions $v\in H^1$.

For the first eigenvalue $\tilde\lambda_1$ of the problem (\ref{Rayleigh}) with the boundary condition (\ref{Antiperiodic})
 one has
$$
\tilde\lambda_1 = \inf_u R[u],
$$
where infimum is taken over the space of $t_0$-antiperiodic functions $v\in H^1$.
\end{prop}
\begin{corollary} For any smooth $t_0$-periodic function $f$ one has the inequality
\label{minmax}
$$
\lambda_0 \leqslant R[f].
$$
For any smooth $t_0$-antiperiodic function $g$ one has the inequality
$$
\tilde\lambda_1 \leqslant R[g].
$$ 
\end{corollary}
\begin{corollary} For the family of the periodic Sturm-Liouville problems (\ref{sturm}) one has $\lambda_0(l) > \lambda_0(l')$
as long\label{monotonicity} as $l>l'$. 
\end{corollary}
\begin{prop} The following inequality holds,
$$
\lambda_0(2)>2.
$$
\label{lambda0(2)}
\end{prop}
\begin{proof} Let us use the variational principle for the problem (\ref{sturm}) with $l=2$,
$$
\lambda_0(2) = \inf_v \dfrac{\displaystyle\int\limits_0^{t_0}\left( 4\pi^2\cos^2\varphi(t)\dot v^2 + \cfrac{4}{\cos^2\varphi(t)}  \right)v^2dt}{\displaystyle\int\limits_0^{t_0}v^2 dt}
\geqslant \inf_v \frac{\displaystyle\int\limits_0^{t_0}\cfrac{4}{\cos^2\varphi(t)}v^2 dt}{\displaystyle\int\limits_0^{t_0}v^2 dt} \geqslant 4 > 2.
$$
\end{proof}

By Theorem \ref{th2}, the functions (\ref{bipolarimmersion2}) are eigenfunctions of the Laplace-Beltrami operator on the $\tilde O_{\frac{p}{q}}$.
It follows from formulae (\ref{bipolarimmersion2}) that the functions $\cos\varphi(t)\sin\theta(t)$ and $\cos\varphi(t)\cos\theta(t)$
are eigenfunctions of the problem (\ref{sturm}) with $l=1$. Proposition \ref{zeroam} implies that both functions have 
exactly $2p$ zeros. Hence, one can set $h_{2p-1}(1,t) = \cos\varphi(t)\sin\theta(t)$ and $h_{2p}(1,t) = \cos\varphi(t)\cos\theta(t)$. 
In the same way $\sin\varphi(t)$ is an eigenfunction of the problem (\ref{sturm}) with $l = 0$ and $\sin\varphi(t)$ has exactly 
$2q$ zeros. Hence, either $h_{2q-1}(0,t) = \sin\varphi(t)$ or $h_{2q}(0,t) = \sin\varphi(t)$. 

It turns out that the most difficult part of this paper is to prove that $h_{2q}(0,t) = \sin\varphi(t)$.
\subsection{Periods of eigenfunctions.}
Suppose that the coefficients $p(t),q(t)$ have a period less than $t_0$.
We are interested in the eigenfunctions with the same period.
\begin{prop} Let $h_i(t)$ be the eigenfuctions of the periodic Sturm-Liuville problem (\ref{Rayleigh},\ref{Periodic})
 with $\dfrac{t_0}{2n}$-periodic 
coefficients enumerated as in Proposition~\ref{SturmProp}. 
Then the $\dfrac{t_0}{2n}$-antiperiodic solutions of the problem (\ref{Rayleigh}) are\label{period}
$h_{2n(2k+1)-1}(t)$ and $h_{2n(2k+1)}(t)$, where $k\in \mathbb{Z}$. 
\end{prop}
\begin{proof} Let us consider the following Sturm-Liouville problem,
\begin{equation*}
\begin{split}
& -\frac{d}{dt}\left(p(t)\frac{dh(t)}{dt}\right) + q(t)h(t)=\lambda h(t),\\
& h(t)\equiv -h\left(t+\frac{t_0}{2n}\right).
\end{split}
\end{equation*}

By Proposition~\ref{SturmProp}, its eigenvalues $\tilde\lambda_i$ form a sequence  
$$
\tilde\lambda_1\leqslant\tilde\lambda_2<\tilde\lambda_3\leqslant\tilde\lambda_4<\ldots
$$

Since $\dfrac{t_0}{2n}$-antiperiodic solutions are also $t_0$-periodic, the corresponding eigenfunctions $\tilde h_i(t)$ are 
solutions of the problem (\ref{Rayleigh},\ref{Periodic}). The eigenfunctions $\tilde h_{2i-1}(t)$ and $\tilde h_{2i}(t)$ have exactly $2i-1$ zeros on the interval
$\left[0,\dfrac{t_0}{2n}\right)$. Hence, they have $2n(2i+1)$ zeros on the interval $[0,t_0)$. There are only two solutions of
 the equation (\ref{Rayleigh},\ref{Periodic}) possesing this quantity of zeroes, therefore $\tilde h_{2i-1}(t)\equiv h_{2n(2i-1)-1}(t)$ and 
$\tilde h_{2i}(t) \equiv h_{2n(2i-1)}$.
\end{proof}

The following proposition can be proved in the same way.
\begin{prop} Let $h_i(t)$ be the eigenfuctions of the periodic Sturm-Liouville problem (\ref{Rayleigh},\ref{Periodic}) 
with $\dfrac{t_0}{n}$-periodic coefficients enumerated as in Proposition~\ref{SturmProp}. Then $\dfrac{t_0}{n}$-periodic
 solutions of the problem (\ref{Rayleigh}) are\label{period2} $h_0$, $h_{2nk-1}$ and $h_{2nk}$, where $k\in \mathbb{Z}$. 
\end{prop}
\subsection{Estimates for $\lambda_{2q-1}(0)$.} The main goal of this section is to prove that the inequality \label{lambda2q}$\lambda_{2q-1}(0)<2$ holds. 
Due to Proposition~\ref{period}, $\lambda_{2q-1}(0)$ is equal to the first eigenvalue of the following problem,
\begin{equation*}
\begin{split}
 -\frac{d}{dt}&\left(4\pi^2\cos^2\varphi(t)\frac{dh(t)}{dt}\right) = \lambda h(t),\\
 h&\left(t+\frac{t_0}{2q}\right)\equiv -h(t).
\end{split}
\end{equation*}

The application of Corollary~\ref{minmax} with $f(t) = \sin \dfrac{2q\pi t}{t_0}$ yields 
\begin{equation*}
\begin{split}
&\lambda_{2q-1}(0)< R[f] =\\&
\frac{16q^2\pi^4\displaystyle\int\limits_0^{\frac{t_0}{2q}}\cos^2\varphi(t)\cos^2\frac{2q\pi t}{t_0}\,dt}{t_0^2
\displaystyle\int\limits_0^{\frac{t_0}{2q}}\sin^2\frac{2q\pi t}{t_0}\,dt} = \frac{32\pi^4q^3}{t_0^3}\int\limits_0^{\frac{t_0}{2q}}\cos^2\varphi(t)
\left(1+\cos\frac{4q\pi t}{t_0}\right)dt.
\end{split}
\end{equation*}

According to Proposition~\ref{zeroam}, the integrand has a symmetry of the form $t\mapsto\dfrac{t_0}{2q}-t$ and 
$\cos\varphi(t)$ is increasing on $\left(0,\dfrac{t_0}{4q}\right)$. Hence, one obtains
\begin{equation*}
\begin{split}
&\int\limits_0^{\frac{t_0}{2q}}\cos^2\varphi(t)\cos\frac{4q\pi t}{t_0}\,dt = \\
&2\int\limits_0^{\frac{t_0}{4q}}\cos^2\varphi(t)
\cos\frac{4q\pi t}{t_0}\,dt < 2\cos^2\varphi\left(\frac{t_0}{8q}\right)\int\limits_0^{\frac{t_0}{4q}}
\cos\frac{4q\pi t}{t_0}\,dt = 0.
\end{split}
\end{equation*}
Therefore, it is sufficient to prove that 
$$
\frac{32\pi^4q^3}{t_0^3}\int\limits_0^{\frac{t_0}{2q}}\cos^2\varphi(t) dt < 2.
$$

It follows from formula~(\ref{phi}) and Proposition~\ref{zeroam} that $\varphi(t)$ is a smooth monotonous function on 
$\left(0,\dfrac{t_0}{2q}\right)$. The obvious equality $t_0 = 2q\int\limits_0^{\frac{t_0}{2q}}\, dt$ holds. 
Thus, using the change of variables $t \to \varphi(t)$ one obtains
$$
\frac{32\pi^4q^3}{t_0^3}\int\limits_0^{\frac{t_0}{2q}}\cos^2\varphi(t) dt = \pi^2\frac{\displaystyle\int\limits_{-b}^b\frac{\cos^5\varphi}
{\sqrt{\cos^4\varphi - \cos^4 b}}\,d\varphi}{\left(\displaystyle\int\limits_{-b}^b\frac{\cos^3\varphi}{\sqrt{\cos^4\varphi - \cos^4 b}}\,d\varphi\right)^3}.
$$
Hence, the question is reduced to estimating this ratio of integrals. Let us denote the numerator by $I_1(b)$ and
the denominator by $I_2(b)$, where $b\in \left[0,\dfrac{\pi}{2}\right]$. We use notations $K\,,E$ 
and $\Pi$ for the complete elliptic integrals of first, second and third kind respectively (see e.g. the book~\cite{Friedman}),
\begin{equation*}
\begin{split}
&K(k) = \int\limits_0^1\frac{1}{\sqrt{1-x^2}\sqrt{1-k^2x^2}}\,dx,\qquad
E(k) = \int\limits_0^1\frac{\sqrt{1-k^2x^2}}{\sqrt{1-x^2}}\,dx, \\
&\Pi(n,k) = \int\limits_0^1\frac{1}{(1-nx^2)\sqrt{1-x^2}\sqrt{1-k^2x^2}}\,dx. 
\end{split}
\end{equation*}
\begin{prop} The function $\dfrac{I_1(b)}{I_2^3(b)}$ is\label{lambda2p} decreasing on $\left(0,\dfrac{\pi}{2}\right)$.
\end{prop}
\begin{proof}
One has
\begin{equation*}
\begin{split}
I_2(b) = 
\int\limits_{-\sin b}^{\sin b}\frac{1-y^2}{\sqrt{(1-y^2)^2-\cos^4b}}dy = 
\int\limits_{-1}^1\frac{1-x^2\sin^2b}{\sqrt{1-x^2}\sqrt{1+\cos^2b - x^2\sin^2b}} dx. 
\end{split}
\end{equation*}
Here the following changes of variables were used, $\sin\varphi = y$, $y=x\sin b$. In the same way
$$
I_1(b) = \int\limits_{-1}^1\frac{(1-x^2\sin^2b)^2}{\sqrt{1-x^2}\sqrt{1+\cos^2b - x^2\sin^2b}}dx.
$$
Let us remark that 
\begin{equation*}
\begin{split}
\frac{d(x\sqrt{1-x^2}\sqrt{1+\cos^2b - x^2\sin^2b})}{dx} =
\frac{3x^4\sin^2b - 4x^2+1+\cos^2b}{\sqrt{1-x^2}\sqrt{1+\cos^2b - x^2\sin^2b}}.
\end{split}
\end{equation*}
Integrating over the interval $[-1,1]$, one obtains the following equality,
$$
\frac{1}{3}\int\limits_{-1}^{1}\frac{3x^4\sin^2b - 4x^2+1+\cos^2b}{\sqrt{1-x^2}\sqrt{1+\cos^2b - x^2\sin^2b}}\,dx = 0.
$$
One can subtract this formula from the $I_1(b)$. Hence, the following equality holds,
$$
I_1(b) = \frac{2}{3}\int\limits_{-1}^1 \frac{(3-\sin^2b-\sin^2b\cos^2b) - 2x^2\sin^2b}{2\sqrt{1-x^2}\sqrt{1+\cos^2b - x^2\sin^2b}}.
$$

Let us introduce the notation $k^2 = \dfrac{\sin^2b}{1+\cos^2b}$. Then, it follows, that
\begin{equation*}
\begin{split}
&I_1(b) = \frac{4}{3}\sqrt{\frac{2}{1+k^2}}\left(E(k) - \frac{(1-k^2)(1+3k^2)}{4(1+k^2)}K(k)\right)\\
&I_2(b) = 2\sqrt{\frac{2}{1+k^2}}\left(E(k) - \frac{1-k^2}{2}K(k)\right).
\end{split} 
\end{equation*}
Since $k(b)$ is an increasing function, it is sufficient to prove that $\frac{I_1}{I_2^3}$ is a decreasing function of $k$. 
Using classical formulae for the derivatives of the elliptic integrals
\begin{equation}
\label{dedk}
\frac{dE(k)}{dk} = \frac{E(k)-K(k)}{k}\,,\qquad \frac{dK(k)}{dk} = \frac{E(k)}{k(1-k^2)} - \frac{K(k)}{k},
\end{equation}
one gets
\begin{equation}
\label{dI2}
\frac{dI_2(k)}{dk} = 2\sqrt{\frac{2}{(1+k^2)^3}}\frac{1-k^2}{2k}(E(k) - K(k))
\end{equation}
and
$$
\frac{dI_1(k)}{dk} = 2\sqrt{\frac{2}{(1+k^2)^3}}\frac{1-k^2}{2k}\left(E(k) - \frac{1+3k^2}{1+k^2}K(k)\right).
$$

Since 
$$
\left(\frac{I_1(k)}{I_2^3(k)}\right)' = \frac{I_1'(k)I_2(k) - 3I_1(k)I_2'(k)}{I_2^4(k)},
$$
it is sufficient to prove that $I_1'(k)I_2(k) - 3I_1(k)I_2'(k)<0$. Using two previous formulae one has
\begin{equation}
\label{formula}
I_1'(k)I_2(k) - 3I_1(k)I_2'(k) = \frac{2(1-k^2)}{k(1+k^2)^2}E(k)\left((1-k^2)K(k) - E(k)\right).
\end{equation}

It is well-known that $K(k)$ is an increasing function. The equality~(\ref{dedk}) implies that $k(1-k^2)\dfrac{dK(k)}{dk} = E(k) - (1-k^2)K(k)>0$.
Hence, the last factor in formula~(\ref{formula}) is negative. Therefore, the right hand side of formula~(\ref{formula}) is negative.
\end{proof}
\begin{corollary} The following inequality holds, $I_2(a)<\dfrac{\pi}{\sqrt{2}}$\label{I2}.
\end{corollary}
\begin{proof} It is well-known that $E(k)<K(k)$. Therefore, according to formula (\ref{dI2}), the function $I_2(b)$ is decreasing.
Thus, one obtains the inequality $I_2(b)<I_2(0) = \dfrac{1}{\sqrt{2}}\displaystyle\int\limits_{-1}^1\dfrac{dx}{\sqrt{1-x^2}} = \dfrac{\pi}{\sqrt{2}}$.
\end{proof}
\begin{prop} For $b\in (0,\frac{\pi}{2})$ the following inequality holds,
$$
\pi^2\frac{I_1(b)}{I_2(b)^3}<2.
$$
\end{prop}
\begin{proof} The statement follows from Proposition \ref{lambda2p} and the equality $\pi^2\dfrac{I_1(0)}{I_2^3(0)} = 2$.
\end{proof}

This completes the proof of the inequality $\lambda_{2q-1}(0)<2$.

\subsection{Proof of the theorem.} It follows from Theorem \ref{th1} that in order to prove Theorem~\ref{MainTheorem} it is 
sufficient to prove that $N(2) = 2q+4p-2$ if $q$ is odd and $N(2)=q+2p-2$ if $q$ is even. 
According to Proposition~\ref{lambda0(2)}, $\lambda_0(2)>2$. 
By Corollary \ref{monotonicity}, one has $\lambda_0(l)>2$ as $l\geqslant 2$. Then due to 
Proposition~\ref{SturmProp}, $\lambda_i(l)>2$, when $l\geqslant 2$ and $k\geqslant 0$. 
For $l=0$, $\lambda_{2q-1}(0)<\lambda_{2q}(0)=2$. Hence, we have $2q$ eigenfunctions of the 
problem~(\ref{sturm}) with $l=0$ with eigenvalues less than $2$. If $q$
 is even, then we need to take into account the invariance under the transformation $(\alpha,t)\mapsto\left(\alpha+\pi, t+\dfrac{t_0}{2}\right)$. The 
application of Proposition~\ref{period2} for $n=2$ leaves $q$ eigenfunctions. By Proposition~\ref{SturmProp} one has
$\lambda_{2k+1}>\lambda_{2k}$. Hence, for $l=1$ the following inequality holds, $\lambda_{2p-2}(1)<\lambda_{2p-1}(1)=\lambda_{2p-2}=2$. In the same 
way one obtains $2p-1$ eigenfunctions if $q$ is odd and $p-1$ eigenfunctions if $q$ is even (here one should apply Proposition~\ref{period}). According to Corollary~\ref
{basis}, any eigenfunction of the problem (\ref{sturm}) with $l\geqslant 1$ provides exactly two eigenfunctions of the 
Laplace-Beltrami operator on $\tilde O_{\frac{p}{q}}$. Thus, if $q$ is odd then one has $N(2) = 2q+2(2p-1) = 2q+4p-2$. If $q$ is
even then one has $N(2) = q+2(p-1) = q+2p-2$.
\subsection{The value of the corresponding functional.}
\begin{prop} If $q$ is odd then $\Lambda_{2q+4p-2}(\tilde O_{\frac{p}{q}}) = 8q\pi I_2(a) < 4\sqrt{2}q\pi$. 
If $q$ is even then $\Lambda_{q+2p-2}(\tilde O_{\frac{p}{q}}) = 4q\pi I_2(a) < 2\sqrt{2}q\pi$.\label{bipOtsuki}
\end{prop}
\begin{proof} By formula~(\ref{h}) for the metric $h$ on the torus $\tilde O_{\frac{p}{q}}$ one has $\sqrt{\det h} = \dfrac{1}{2\pi}$. 
Hence the corresponding value of the functional is equal to 
$$
2\int\limits_0^{t_0}\int\limits_0^{2\pi}\frac{1}{2\pi}\,d\varphi\,d\alpha = 2t_0.
$$
For even $q$ one need to take into account the invariance under the transformation $(\alpha,t)\mapsto\left(\alpha+\pi, t+\dfrac{t_0}{2}\right)$, so this value has to be divided by two. Arguing as in Section~\ref{lambda2q} one obtains that
$$
t_0 = 2q\int\limits_0^{\frac{t_0}{2q}}\,dt = 4q\pi I(b).
$$
The application of Corollary~\ref{I2} yields the desired inequalities.
\end{proof}
\subsection{Proof of Proposition~\ref{xi}.} Let us\label{dxi} consider $\Xi(b)$ as a function of $b\in\left(0,\dfrac{\pi}{2}\right)$.
One has to prove that $\Xi(b)$ is decreasing on this interval. Let us begin with expressing $\Xi(b)$ in terms of elliptic 
integrals,
\begin{equation*}
\begin{split}
&\cos^2b\int\limits_{-b}^b\frac{1}{\cos\varphi\sqrt{\cos^4\varphi-\cos^4b}}\,d\varphi = 
\cos^2b\int\limits_{-\sin b}^{\sin b}\frac{1}{(1-y^2)\sqrt{(1-x^2)^2-\cos^4b}}\,d\varphi \\ &= 
\cos^2b\int\limits_{-1}^{1}\frac{1}{(1-x^2\sin^2b)\sqrt{1-x^2}\sqrt{1+\cos^2b - x^2\sin^2b}} = \\
&2\frac{\cos^2b}{\sqrt{1+\cos^2b}}\,\Pi\left(\sin^2b,\frac{\sin b}{\sqrt{1+\cos^2b}}\right) = 
2\frac{1-n}{\sqrt{2-n}}\,\Pi\left(n,\sqrt{\frac{n}{2-n}}\right),
\end{split}
\end{equation*}
where $n=\sin^2b$. Here the following changes of variables were used, $\sin\varphi = y$, $y = x\sin b$. Now the equality $\lim\limits_{b\to 0}\Xi(b)=\dfrac{\sqrt{2}}{2}\pi$ follows from sustituting $n=0$ 
into this formula. Using the formulae for the derivatives of $\Pi(n,k)$,
\begin{equation*}
\begin{split}
&\frac{\partial\Pi(n,k)}{\partial n} = \frac{1}{2(k^2-n)(n-1)}\left(E(k) +\frac{1}{n}(k^2-n)K(k) + \frac{1}{n}(n^2-k^2)\Pi(n,k)\right), \\
&\frac{\partial\Pi(n,k)}{\partial k} = \frac{k}{n-k^2}\left(\frac{E(k)}{k^2-1} + \Pi(n,k)\right),
\end{split}
\end{equation*}
one obtains
$$
\frac{d\Xi(n)}{dn} = \frac{1}{2n\sqrt{2-n}}\left(E\left(\sqrt{\frac{n}{2-n}}\right) - K\left(\sqrt{\frac{n}{2-n}}\right)\right) <0.
$$

\subsection*{Acknowledgments}
The author thanks A.V. Penskoi for statement of this problem, fruitful discussions and invaluable help in the preparation
of the manuscript. 


\begin{thebibliography}{99}
\bibitem{Berger} M. Berger, Sur les premi\`eres valeurs propres des var\'et\'es Riemanniennes, \textit{Compositio Math.} \textbf{26} (1973), 129-149.
\bibitem{Friedman} P. Byrd, M. Friedman. \textit{Handbook of Elliptic Integrals for Engineers and Scientists,} Springer Verlag,
New York-Heidelberg-Berlin, 1971.
\bibitem{KoddingtonLevinson} E. A. Coddington, N. Levinson. \textit{Theory of ordinary differential equations,} McGraw-Hill Book Company, Inc., New York-Toronto-London, 1955.
\bibitem{ColboisDodziuk} B. Colbois, J. Dodziuk, Riemannian metrics with large $\lambda_1$, \textit{Proc. Amer. Math. Soc.} \textbf{122}:3 (1994), 905-906.
\bibitem{EGJ} A. El Soufi, H. Giacomini, M. Jazar, A unique extremal metric for the least eigenvalue
of the Laplacian on the Klein bottle. \textit{Duke Math. J.} \textbf{135}:1 (2006), 181-202. Preprint
\texttt{arXiv:math/0701773}.
\bibitem{ElSoufiIlias1} A. El Soufi, S. Ilias, Laplacian eigenvalues functionals and metric deformations on compact manifolds. \textit{J. Geom. Phys.} \textbf{58}:1 (2008), 89-104. Preprint \texttt{arXiv:math/0701777}.
\bibitem{ElSoufiIlias2} A. El Soufi, S. Ilias, Riemannian manifolds admitting isometric immersions by their first eigenfunctions. \textit{Pacific. J. Math.} \textbf{195}:1 (2000), 91-99.
\bibitem{Henrot} A. Henrot. \textit{Extremum problems for eigenvalues of elliptic operators,} Birkh\"auser Verlag, Basel-Boston-Berlin, 2006.
\bibitem{HsiangLawson} W.-Y. Hsiang, H. B. Lawson, Minimal submanifolds of low cohomogeneity. \textit{J. Diff. Geom.} \textbf{5} (1971), 1-38.
\bibitem{JNP} D. Jakobson, N. Nadirashvili, I. Polterovich, Extremal Metric for the First Eigenvalue on a
Klein Bottle. \textit{Canad. J. Math.} \textbf{58}:2 (2006), 381-400. Preprint \texttt{arXiv:math/0311484}.
\bibitem{Kenmoutsu} K. Kenmotsu, A characterization of bipolar minimal surfaces in $\mathbb{S}^4$. \textit{T\^ohoku Math. J.} \textbf{26} (1974), 587-598.
\bibitem{Kobayasi} S. Kobayashi, K. Nomizu, \textit{Foundations of differential geometry,} Vol. II, Interscience Publishers John Wiley \& Sons, Inc., New York-London-Sydney, 1969.
\bibitem{Korevaar} N. Korevaar, Upper bounds for eigenvalues of conformal metrics. \textit{J. Differential Geom.} \textbf{37}:1 (1993), 79-93.
\bibitem{Lapointe} H. Lapointe, Spectral properties of bipolar minimal surfaces in $\mathbb{S}^4$. \textit{Differential Geom. Appl.} \textbf{26}:1 (2008), 9-22. Preprint \texttt{arXiv:math/0511443}.
\bibitem{Lawson} H. B. Lawson, Complete minimal surfaces in $\mathbb{S}^3$. \textit{Ann. of Math.} \textbf{92} (1970), 335-374.
\bibitem{Nadirashvilli} N. Nadirashvili, Berger's isometric problem and minimal immersions of surfaces. \textit{Geom. Funct. Anal} \textbf{6}:5 (1996), 877-897.
\bibitem{Otsuki} T. Otsuki, Minimal hypersurfaces in a Riemannian manifolds of constant curvature. \textit{Amer. J. Math.} \textbf{92}:1 (1970), 145-173.
\bibitem{PenskoiOtsuki} A. V. Penskoi, Extremal spectral properties of Otsuki tori. Preprint \texttt{arXiv:1108.5160}.
\bibitem{PenskoiLawson} A. V. Penskoi, Extremal spectral properties of Lawson tau-surfaces and the Lam\'e equation. \textit{Moscow Math. J.} \textbf{12}:1 (2012), 173-192. Preprint \texttt{arXiv:1009.0285}.
\bibitem{YangYao} P. C. Yang, S.-T. Yau, Eigenvalues of the laplacian of compact Riemann surfaces and minimal submanifolds. \textit{Ann. Scuola Norm. Sup. Pisa Cl. Sci.} \textbf{7}:1 (1980), 55-63.
\end{thebibliography}
\end{document}